\documentclass[preprint,11pt]{elsarticle}     %12pt

\usepackage{float}
\usepackage{amstext,amsfonts,amsmath,amssymb,array}
\usepackage{amsthm,moresize}
\usepackage{graphics,graphicx,color}
\makeatletter
\def\@author#1{\g@addto@macro\elsauthors{\normalsize%
\def\baselinestretch{1}%
\upshape\authorsep#1\unskip\textsuperscript{%
\ifx\@fnmark\@empty\else\unskip\sep\@fnmark\let\sep=,\fi
\ifx\@corref\@empty\else\unskip\sep\@corref\let\sep=,\fi}%
\def\authorsep{\unskip,\space}%
\global\let\@fnmark\@empty
\global\let\@corref\@empty  %% Added
\global\let\sep\@empty}%
\@eadauthor={#1}
}
\makeatother

\DeclareMathOperator{\spec}{Spec}

\newtheorem{theorem}{Theorem}[section]

\newtheorem{corollary}[theorem]{Corollary}
\newtheorem{lemma}[theorem]{Lemma}
\theoremstyle{definition}
\newtheorem{definition}{Definition}[section]

\newtheorem{remark}[theorem]{Remark}

\renewcommand\ell{l}

\journal{Journal}

\begin{document}

\begin{frontmatter}

\title{\Large Total graph of a signed graph}

\author[1]{Francesco Belardo}
\ead{fbelardo@unina.it}
\author[2]{Zoran Stani\'c}
\ead{zstanic@matf.bg.ac.rs}
\author[3]{Thomas Zaslavsky}
\ead{zaslav@math.binghamton.edu}

\address[1]{\scriptsize Department of Mathematics and Applications, University of Naples Federico II, I-80126 Naples, Italy}
\address[2]{\scriptsize Faculty of Mathematics, University of Belgrade, Studentski trg 16, 11 000 Belgrade, Serbia}
\address[3]{\scriptsize Department of Mathematical Sciences, Binghamton University, Binghamton, NY 13902-6000, United States}

\begin{abstract}
The total graph is built by joining the graph to its line graph by means of the incidences. We introduce a similar construction for signed graphs. Under two similar definitions of the line signed graph, we define the corresponding total signed graph and we show that it is stable under switching. We consider balance, the frustration index and frustration number, and the largest eigenvalue. In the regular case we compute the spectrum of the adjacency matrix of the total graph and the spectra of certain compositions, and we determine some with exactly two main eigenvalues.
\end{abstract}


\end{frontmatter}

\section{Introduction}

We define the total graph of a signed graph in a way that extends to signed graphs the spectral theory of ordinary total graphs of graphs.
The usual total graph is built by joining the graph to its line graph by means of its vertex-edge incidences; this construction coordinates well with the adjacency matrix. When we consider signed graphs, there is a similar definition which extends the notion of line graph of a signed graph and which coordinates combinatorial and matrix constructions. Working from two similar definitions of the line signed graph we define the corresponding total graphs, and we show they are stable under switching. We examine fundamental properties of the signed total graph including balance, the degree of imbalance as measured by the frustration index and frustration number, and the largest eigenvalue. In the regular case we compute the spectrum of the adjacency matrix and the spectra of certain compositions and determine some with exactly two main eigenvalues.

A \textit{signed graph} $\Sigma$ is a pair $(G, \sigma)=G_{\sigma}$, where $G=(V, E)$ is an ordinary (unsigned) graph, called the \textit{underlying graph}, and $\sigma\colon E\longrightarrow\{-1, +1\}$ is the \textit{sign function} (the signature). The edge set of a signed graph is composed of subsets of positive and negative edges. We interpret an unsigned graph $G$ as the {\em all-positive} signed graph $(G,+)=+G$, whose signature gives $+1$ to all the edges. Similarly, by $(G,-)=-G$ we denote a graph $G$ with the {\em all-negative signature}. In general, we have $-\Sigma=(G,-\sigma)$.

Many familiar notions about unsigned graphs extend directly to signed graphs. For example, the degree $d(v)$ of a vertex $v$ in $G_{\sigma}$ is simply its degree in $G$. On the other hand, there also are notions exclusive to signed graphs, most importantly the sign of a cycle, namely the product of its edge signs.
A signed graph or its subgraph is called {\it balanced} if every cycle in it, if any, is positive.
Oppositely, $\Sigma$ is \emph{antibalanced} if $-\Sigma$ is balanced, i.e., every odd (resp.,~even) cycle of $\Sigma$ is negative (resp.,~positive), e.g., if $\Sigma=-G$.
Balance is a fundamental concept of signed graphs and measuring how far a signed graph deviates from it is valuable information.  The \textit{frustration index} $\ell$ (resp.,~the \textit{frustration number} $\nu$) of a signed graph is the minimum number of edges (resp.,~vertices) whose removal results in a balanced signed graph.  These numbers generalize the edge biparticity and vertex biparticity of a graph $G$, which equal $\ell(-G)$ and $\nu(-G)$, respectively.

Also important is the operation of switching.  If $U$ is a set of vertices of $\Sigma$, the switched signed graph $\Sigma^U$ is obtained from $\Sigma$ by reversing the signs of the edges in the cut $[U,\Sigma\setminus U]$. The signed graphs $\Sigma$ and $\Sigma^U$ are said to be {\it switching equivalent}, written $\Sigma \sim \Sigma^U$,  and the same is said for their signatures.
Notably, a signed graph is balanced if and only if it is switching equivalent to the all-positive signature \cite{zas2} and it is antibalanced if and only if it switches to the all-negative signature. For basic notions and notation on signed graphs not given here we refer the reader to \cite{zas2, Zas}.

The \textit{adjacency matrix} $A_{\Sigma}$ of $\Sigma=G_\sigma$ is obtained from the standard $(0,1)$-adjacency matrix of $G$ by reversing the sign of all $1$'s which correspond to negative edges.  The \textit{eigenvalues} of $\Sigma$ are identified to be the eigenvalues of~$A_{\Sigma}$; they form the \textit{spectrum} of $\Sigma$. The \textit{Laplacian matrix} of~$\Sigma$ is defined by $L_{\Sigma}=D_G-A_{\Sigma}$, where $D_G$ is the diagonal matrix of vertex degrees of $G$. Analogously, the \textit{Laplacian eigenvalues} of $\Sigma$ are the eigenvalues of $L_{\Sigma}$.

It is well known that the signed graphs $\Sigma=G_{\sigma}$ and $\Sigma'=G_{\sigma'}$ are switching equivalent if and only if there exists a diagonal matrix $S$ of $\pm1$'s, called the {\em switching matrix}, such that $A_{\Sigma'}=S^{-1}A_{\Sigma}S$, and we say that the corresponding matrices are \textit{switching similar}. More generally, the signed graphs $\Sigma$ and $\Sigma'$ are \textit{switching isomorphic} if there exist a permutation matrix $P$ and a switching matrix $S$ such that $A_{\Sigma'}=(PS)^{-1}A_{\Sigma}(PS)$; in fact, $PS$ can be seen as a signed permutation matrix (or a $\{1, 0, -1\}$-monomial matrix).

The frustration index and the frustration number are among the most investigated invariants -- for more details one can consult~\cite{Zas}. Similarly, the largest eigenvalue of the adjacency matrix is  the most investigated spectral invariant of graphs.
Evidently, switching preserves the eigenvalues of  $A_{\Sigma}$ and $L_{\Sigma}$, and it also preserves the signs of cycles, so that switching equivalent signed graphs share the same set of positive (and negative) cycles and have the same frustration index and frustration number. For the above reasons, when we consider a signed graph $\Sigma$ perspective, we are considering its switching isomorphism class $[\Sigma]$, and we focus our attention to the properties of $\Sigma$ which are invariant under switching isomorphism.

Here is the remainder of the paper. In Section~2 we discuss the concept of line graph of a signed graph. The total graph is presented in Section~\ref{sec:Gen}. In Section~\ref{sec:Reg} we consider regular underlying graphs and, similarly to what is known for unsigned graphs \cite{CvDS}, we give the eigenvalues of a total graph of a signed graph by means of the eigenvalues of its root signed graph.

%A note on terminology:
We stress that a line (total) graph of a signed graph is always signed, so ``line (total) graph'' of $\Sigma$ means the same as ``signed line (total) graph'' of $\Sigma$.
For brevity we also call these graphs ``line (total) signed graphs'' and ``signed line (total) graphs'' (although literally the latter can mean any signature on an unsigned line or total graph; indeed an entirely different signed total graph has been defined by Sinha and Garg~\cite{sinha}).

\section{Line Graph(s)}

The line graph is a well-known concept in graph theory: given a graph $G=(V(G),E(G))$, the line graph $\mathcal{L}(G)$ has $E(G)$ as vertex set, and two vertices of $\mathcal{L}(G)$ are adjacent if and only if the corresponding edges are adjacent in $G$. If we consider signed graphs $\Sigma=G_{\sigma}$, then a (signed) line graph $\mathcal{L}(G_{\sigma})$ should have $\mathcal{L}(G)$ as its underlying graph. However, what signature should we associate to it? The answer to this question is a matter of discussion because it is possible to have several very different signatures. In this section we shall consider the two relevant ones defined in the literature.

\subsection{Definitions of a line graph}

Zaslavsky gave the first definition of incidence matrix of signed graphs \cite{zas2}, which is a necessary step in a spectrally consistent definition of a line graph.
For a signed graph $\Sigma=G_{\sigma}$, we introduce the vertex-edge {\em orientation} $\eta\colon V(G)\times E(G) \longrightarrow \{1, 0, -1\}$ formed by obeying the following rules:
\begin{enumerate}[(O1)]
\item $\eta(i,jk)=0$ if $i\notin\{j, k\}$;
\item $\eta(i, ij)=1$ or $\eta(i, ij)=-1$;
\item $\eta(i, ij)\eta(j, ij)=-\sigma(ij)$.
\end{enumerate}
(The minus sign in (O3) is necessary for several purposes, such as with signed-graph orientations \cite{ZasC, ZasO} and geometry \cite{ZasO}.)
The {\em incidence matrix} $B_\eta=(\eta_{ij})$ is a vertex-edge incidence matrix derived from $G_\sigma$, such that its $(i,e)$-entry is equal to $\eta(i,e)$. However, it is not uniquely determined by $\Sigma$ alone. As in the definition of the oriented incidence matrix for unsigned graphs, one can randomly choose an entry $\eta(i, ij)$ to be either $+1$ or $-1$, but the entry $\eta(j, ij)$ is then determined by $\sigma(ij)$, so $\eta$ is called an {\em orientation of $G_\sigma$} (and a {\em biorientation of $G$}, the unsigned underlying graph). Zaslavsky later interpreted $B_\eta$ as the incidence matrix of an oriented signed graph \cite{ZasO} and recognized that the same was an alternate definition of {\em bidirected graphs} as in~\cite{edmonds}. From a signed graph $\Sigma$ we get many bidirected graphs $\Sigma_{\eta}$, but each of them leads back to the same signed graph $\Sigma$.

Let $A^\intercal$ denote the transpose of the matrix $A$. The incidence matrix has an important role in spectral theory. The Laplacian matrix can be derived as the row-by-row product of the matrix $B_\eta$ with itself:
\[B_\eta B_\eta^\intercal = L_{\Sigma}.\]
Notably, regardless of the orientation $\eta$ chosen, we get the same $L_{\Sigma}$. It is well known that the column-by-column product of $B_\eta$ with itself is a matrix sharing the nonzero eigenvalues with the row-by-row product. This was one motive for  Zaslavsky \cite{ZasL, Zas} to define the line graph of a signed graph as the signed graph $\mathcal{L}_C(\Sigma)=(\mathcal{L}(G),\sigma_C)$ whose signature $\sigma_C$ is determined by the  adjacency matrix is $A_{\mathcal{L}_C(\Sigma)}$ defined here:
\begin{equation}\label{eq:BTBZ}
A_{\mathcal{L}_C(\Sigma)}=2I-B_\eta^\intercal B_\eta.
\end{equation}
Unlike in the case of the Laplacian matrix of $\Sigma$, the matrix $A_{\mathcal{L}_C(\Sigma)}$ does depend on the orientation~$\eta$. On the other hand, choosing a different orientation  $\eta'$ of $\Sigma$ leads to a matrix $A_{(\mathcal{L}(G),\sigma')}$ that is switching similar to $A_{(\mathcal{L}_C(G),\sigma)}$ (cf.\ \cite{Zas}). Hence, $A_{(\mathcal{L}(G),\sigma)}$ defines a line graph up to switching similarity, so it can be used for spectral purposes.
One of the benefits of this definition is that the line graph of a signed graph with an all-negative signature is a line graph with an all-negative signature. In other words, if $-G$ is a graph $G$ whose edges are taken negatively, we get
\[\mathcal{L}_C(-G)=-\mathcal{L}(G).\]
The above fact has two evident consequences. The first one is that the iteration of the (Zaslavsky) line graph operator always give a signed graph with all-negative signature, namely $\mathcal{L}_C^{(k)}(-G)=-\mathcal{L}^{(k)}(G)$. The second one is that if we map simple unsigned graphs to the theory of signed graphs as signed graphs with the all-negative signature (instead of the all-positive, as stated in the introduction), then Zaslavsky's line graph is a direct generalization of the usual line graph defined for unsigned graphs.  We shall call this line graph the \emph{combinatorial line graph} of $\Sigma$.

However, from a spectral viewpoint, the fact that the matrix $A_{\mathcal{L}_C(\Sigma)}$ has spectrum in the real interval $(-\infty,2]$, in contrast with the usual concept of spectral graph theory for which a line graph has spectrum in the real interval $[-2,+\infty]$. Hence, the authors of \cite{BeSi} decided to modify Zaslavsky's definition to
\begin{equation}\label{eq:BTBS}
 A_{(\mathcal{L}_S(G),\sigma)}=B_\eta^\intercal B_\eta-2I.
 \end{equation}
In fact, the two definitions are virtually equivalent, as $\mathcal{L}_C(\Sigma)=-\mathcal{L}_S(\Sigma)$.  Moreover, they can be used for different purposes. The latter definition is tailored for those spectral investigations in which an unsigned graph is considered as a signed one with all-positive signature. Clearly, in this case its adjacency (and Laplacian) matrix remains unchanged and the spectral theory of unsigned graphs can easily be encapsulated into the spectral theory of signed graphs. For example, in this case  $\mathcal{L}_S(\Sigma)$ is coherent with the Laplacian and signless Laplacian spectral theories of unsigned graphs, and it can be used to investigate their spectra (cf.\ \cite{BePiSi}). Also, Hoffmann's theory of generalized line graph \cite{Hoff} fits well with the signature of $\mathcal{L}_S(\Sigma)$. For these reasons, we shall call $\mathcal{L}_S(\Sigma)$ the {\em spectral line graph}.

In Fig.~\ref{fig:G} we illustrate an example of a signed graph $\Sigma$, an orientation $\Sigma_\eta$ and the consequent line graph $\mathcal{L}_C(\Sigma)$. Here and later, positive edges are represented by solid lines and negative edges are represented by dotted lines.

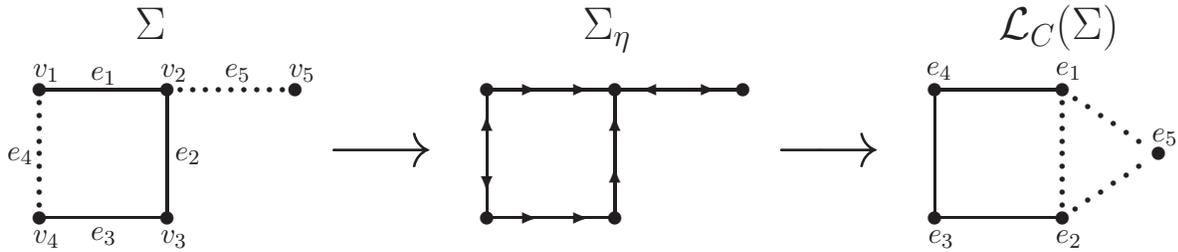
\begin{figure}[H]	\centering
\unitlength=0.85mm
\begin{picture}(175,40)(0,10)\label{Figura1}
 \thicklines

\put(0,15){\circle*{2.0}} \put(0,35){\circle*{2.0}} \put(20,15){\circle*{2.0}} \put(20,35){\circle*{2.0}} \put(40,35){\circle*{2.0}}

\put(70,15){\circle*{2.0}} \put(70,35){\circle*{2.0}} \put(90,15){\circle*{2.0}} \put(90,35){\circle*{2.0}} \put(110,35){\circle*{2.0}}

\put(140,15){\circle*{2.0}} \put(140,35){\circle*{2.0}} \put(160,15){\circle*{2.0}} \put(160,35){\circle*{2.0}} \put(175,25){\circle*{2.0}}

\put(0,15){\line(1,0){20}} \put(0,35){\line(1,0){20}} \put(140,15){\line(0,1){20}}  \put(20,15){\line(0,1){20}}

\multiput(0,15)(0,2){11}{\circle*{0.7}} \multiput(20,35)(2,0){11}{\circle*{0.7}}

\put(140,15){\line(1,0){20}}  \put(140,35){\line(1,0){20}}

%\put(160,35){\line(3,-2){15}}  \put(160,15){\line(3,2){15}}
\multiput(160,35)(2,-1.333){7}{\circle*{0.7}}
\multiput(160,15)(2,1.333){7}{\circle*{0.7}}
\multiput(160,15)(0,2){11}{\circle*{0.7}}

\put(70,15){\line(1,0){20}} \put(70,15){\vector(1,0){16}} \put(70,15){\vector(1,0){8}}
\put(90,15){\line(0,1){20}} \put(90,15){\vector(0,1){8}} \put(90,15){\vector(0,1){16}}
\put(70,35){\line(1,0){20}} \put(70,35){\vector(1,0){8}} \put(70,35){\vector(1,0){16}}
\put(70,15){\line(0,1){20}} \put(70,15){\vector(0,1){16}} \put(70,35){\vector(0,-1){16}}
\put(90,35){\line(1,0){20}} \put(90,35){\vector(1,0){16}} \put(110,35){\vector(-1,0){16}}

\put(45,23){\Huge{$\longrightarrow$}} \put(115,23){\Huge{$\longrightarrow$}}

\put(-1,37){$v_1$} \put(19,37){$v_2$} \put(39,37){$v_5$} \put(-1,11){$v_4$} \put(19,11){$v_3$}

\put(8,36.5){$e_1$} \put(21.2,24){$e_2$} \put(-5,24){$e_4$} \put(8,11.5){$e_3$} \put(29,37){$e_5$}

\put(138.7,37.3){$e_4$} \put(159,37.3){$e_1$} \put(174,27){$e_5$} \put(139,11){$e_3$} \put(159,11){$e_2$}

\put(15,43){\LARGE{$\Sigma$}} \put(85,43){\LARGE{$\Sigma_{\eta}$}} \put(150,43){\LARGE{$\mathcal{L}_C(\Sigma)$}}

\end{picture}
	\caption{A signed graph, an orientation and the combinatorial line graph.}\label{fig:G}
\end{figure}

The matrix definitions of line graphs have combinatorial analogs; in fact, Zaslavsky's original definition of the line graph of a signed graph (even prior to \cite{ZasL}) was combinatorial.
For $\Sigma=G_\sigma$ the underlying graph of $\mathcal{L}_*(\Sigma)$ is the line graph $\mathcal{L}(G)$, while the sign of the edge $ef$ ($e, f$ being the edges of $\Sigma$ with a common vertex $v$) is
\begin{equation}\label{eq:lgor}
\sigma(ef)=\left\{\begin{array}{rl}
	-\eta(ve)\eta(vf)& \text{for}~ *=C,\\
	\eta(ve)\eta(vf)& \text{for}~ *=S.
\end{array}\right.
\end{equation}
Indeed, we may orient the line graph by defining $\eta_{\mathcal{L}}(e,ef) = \eta(v,e)$ for two edges $e,f$ with common vertex $v$ in $\Sigma$ \cite{Zas}.  Then the combinatorial definition of line graph signs follows the rule (O3).

\begin{remark}
Which is the best definition of a line graph of a signed graph? Zaslavsky prefers the one defined in \eqref{eq:BTBZ} because it is consistent with the basic relationship between signs and orientation stated in (O3). Belardo and Stani\'c instead prefer \eqref{eq:BTBS} because it is the one coherent with existing spectral graph theory and it is more prominent in the literature. This led to a long discussion among the three authors of this manuscript and the late Slobodan Simi\'c. In the end, we recognize the validity of each definition, since either variant can be used and, after all, they are easily equivalent.
\end{remark}

\begin{remark}\label{rem:multi}
The identities \eqref{eq:BTBZ} and \eqref{eq:BTBS} remain valid even if $\Sigma$ contains multiple edges. A pair of edges located between the same pair of vertices form a \textit{digon}, i.e., a 2-vertex cycle, which is positive if and only if the edges share the same sign.
Here the matrix definition diverges from the combinatorial definition.  In the combinatorial definition, a digon in $\Sigma$ with edges $e$ and $f$ gives rise to a digon in the line graph having the same sign as the original digon; aside from signs, this is as in the unsigned line graph.
In the matrix definitions of $\mathcal{L}_*(\Sigma)$ it gives rise to a pair of non-adjacent vertices if the corresponding digon is negative and a pair joined by two parallel edges of the same sign if the corresponding digon is positive; this is consistent with the fact that a negative digon in $\Sigma$ disappears in the adjacency matrix $A_\Sigma$.  Zaslavsky calls this kind of line graph, where parallel edges of opposite sign cancel each other, \emph{reduced}.
Therefore, an unreduced line graph has no multiple edges if and only if $\Sigma$ has no digons, and a reduced line graph has no multiple edges if and only if $\Sigma$ has no positive digons.
%Such a $\Sigma$ is called by Zaslavsky a \textit{simply signed graph}.
\end{remark}

\subsection{Properties of line graphs}

Here are some observations that follow directly from \eqref{eq:BTBZ} and \eqref{eq:BTBS}. All triangles that arise from a star of $\Sigma$ are negative (resp.,~positive) in $\mathcal{L}_C(\Sigma)$ (resp.,~$\mathcal{L}_S(\Sigma)$). Every cycle of $\Sigma$ keeps its signature in $\mathcal{L}_C(\Sigma)$. Every even cycle of $\Sigma$ keeps its signature in $\mathcal{L}_S(\Sigma)$ and every odd cycle of $\Sigma$ reverses its signature in $\mathcal{L}_S(\Sigma)$.

\begin{theorem}\label{balanced_line_graph}
Let $G$ be an unsigned graph. Then $-\mathcal{L}_{C}(-G)$ and $\mathcal{L}_{S}(-G)$ are balanced signed graphs, and therefore switching equivalent to $+\mathcal{L}(G)$.
\end{theorem}
\begin{proof}[First Proof.]
Recall that a balanced graph has no negative cycles. If we consider the line graph $\mathcal{L}(G)$, we distinguish three types of cycle: (i) those that arise from the cycles of $G$, (ii) those that arise from induced stars in $G$ (forming cliques), (iii) those obtained by combining the types (i) and (ii).

Let us consider $-\mathcal{L}_{C}(-G)$ and $\mathcal{L}_{S}(-G)$. We have to prove that $\mathcal{L}_{C}(-G)$ is antibalanced, or equivalently that $\mathcal{L}_{S}(-G)$ is balanced. In $\mathcal{L}_{S}(\Sigma)$ the signed cycles of type (i), originating from cycles $C_k$ of $\Sigma$, get the sign $(-1)^k\sigma(C_k)$. Hence, they are transformed into positive cycles of $\mathcal{L}_S(\Sigma)$ if and only if $\Sigma\sim-G$. Consider next the cycles of type (ii). The cliques $(K_t,\sigma)$ of $\mathcal{L}_S(\Sigma)$, originating from an induced $K_{1,t}$ of $G$, are switching equivalent to $+K_n$. To see the latter, without loss of generality one can choose the biorientation of $K_{1,n}$ for which the  vertex-edge incidence at the center of the star is positive (the arrows are inward directed), so the obtained clique is indeed $+K_n$. These cycles are always positive, regardless of the signature of $\Sigma$. Finally, for the cycles of type (iii), we know from \cite{zas4} that the signs of a set of cycles that span the cycle space determine all the signs. Hence, the cycles of type (iii) are positive if and only if the cycles of type (i) are positive, that is, $\Sigma=-G$.
\end{proof}
\begin{proof}[Second Proof.]
Choose the orientation for $-G$ in which $\eta(v,e)=+1$ for every incident vertex and edge.  Then $\mathcal{L}_{C}(-G)$ is easily seen to be all negative by the combinatorial definition \eqref{eq:lgor} of edge signs, thus $\mathcal{L}_{C}(-G) = -\mathcal{L}(G)$, which is antibalanced.  Then, $\mathcal{L}_{S}(-G) = -\mathcal{L}_{C}(-G) = +\mathcal{L}(G)$, which is balanced.  Choosing a different orientation for $-G$ has the effect of switching both line graphs, so it does not change the state of balance or antibalance.
\end{proof}

We conclude this section by analysing balance and the degree of imbalance of line graphs.
Because $\mathcal{L}_{S}(\Sigma)=-\mathcal{L}_{C}(\Sigma)$, balance of the combinatorial line graph $\mathcal{L}_{C}(\Sigma)$ is equivalent to antibalance of the spectral line graph $\mathcal{L}_{S}(\Sigma)$, and balance of the latter is equivalent to antibalance of the former.

\begin{theorem}\label{the:lg} Let $\Sigma=G_{\sigma}$ be a signed graph of order $n$ and size $m$. The following hold true:
	\begin{enumerate}[\rm(i)]

		\item\label{the:lgCb} $\mathcal{L}_{C}(\Sigma)$ is balanced (and $\mathcal{L}_{S}(\Sigma)$ is antibalanced) if and only if $\Sigma$ is a disjoint union of paths and positive cycles.
		
		\item\label{the:lgSb} $\mathcal{L}_{S}(\Sigma)$ is balanced (and $\mathcal{L}_{C}(\Sigma)$ is antibalanced) if and only if $\Sigma$ is antibalanced.
				
		\item\label{the:lgSnu}  $\ell(\mathcal{L}_{S}(\Sigma)) \geq \nu(\mathcal{L}_{S}(\Sigma)) = \ell(-\Sigma)$.
		
		\item\label{the:lgfi} $\ell(\mathcal{L}_{C}(\Sigma))\geq \sum_{v\in V(\Sigma)} \big\lfloor \frac{(d(v)-1)^2}{4} \big\rfloor$.
		
	\end{enumerate}
\end{theorem}

\begin{proof}
\eqref{the:lgCb}  Balance of $\mathcal{L}_{C}(\Sigma)$ means that $\Sigma$ does not contain a vertex of degree 3 or greater, as the corresponding edges produce negative triangles. Evidently, $\Sigma$ cannot contain negative cycles, because this leads to negative cycles in $\mathcal{L}_{C}(\Sigma)$. If $\Sigma$ is a disjoint union of paths and positive cycles, then $\mathcal{L}_{C}(\Sigma)$ is again a disjoint union of paths and positive cycles.

\eqref{the:lgSb}  A line graph $\mathcal{L}(G)$ has three kinds of cycle.  A vertex triangle corresponds to three edges incident with a single vertex of $G$; all vertex triangles are negative in $\mathcal{L}_{C}(\Sigma)$.  A line cycle is derived from a cycle $C$ in $G$ and has the same sign in $\mathcal{L}_{C}(\Sigma)$.  The remaining cycles are obtained by concatenation of cycles of the first two kinds.  It follows that $\mathcal{L}_{C}(\Sigma)$ is antibalanced if and only if every line cycle is antibalanced, thus if and only if $\Sigma$ is antibalanced.

\eqref{the:lgSnu}  An edge set $A$ in $-\Sigma$ is a set $A$ of vertices in $\mathcal{L}_{S}(\Sigma)$; and $\mathcal{L}_{S}(\Sigma) \setminus A = \mathcal{L}_{S}(\Sigma \setminus A)$.  By \eqref{the:lgSb}, $\mathcal{L}_{S}(\Sigma \setminus A)$ is balanced if and only if $-\Sigma \setminus A$ is balanced.  Thus, the smallest size of an edge set $A$ such that $-\Sigma \setminus A$ is balanced, which is $\ell(-\Sigma)$, equals the smallest size of a vertex set $A$ such that $\mathcal{L}_{S}(\Sigma \setminus A)$ is balanced, which is $\nu(\mathcal{L}_{S}(\Sigma))$.

The inequality follows from the general observation that $\ell \geq \nu$ for every signed graph.

\eqref{the:lgfi}   A vertex of degree $d(v)$ in $\Sigma$ generates in $\mathcal{L}_{C}(\Sigma)$ an antibalanced vertex clique $-K_{d(v)}$.  An edge set $B$ in the line graph such that $\mathcal{L}_{C}(\Sigma) \setminus B$ is balanced must contain enough edges to make each such clique balanced; this number is $\ell(-K_{d(v)}) = \big\lfloor \frac{(d(v)-1)^2}{4} \big\rfloor$, obtained by dividing the vertices of $K_{d(v)}$ into two nearly equal sets and deleting the edges within each set.  Each edge of the line graph is in only one vertex clique, so the minimum number of edges required to balance every vertex clique is the sum of these quantities.  That proves the inequality.
\end{proof}

We do not expect equality in part (iv) because deleting the edges as in the proof may not eliminate all negative cycles in $\mathcal{L}_{C}(\Sigma)$. The problem of finding an exact formula for $\ell(\mathcal{L}_{C}(\Sigma))$ or $\ell(\mathcal{L}_{S}(\Sigma))$ in terms of $\Sigma$, or even a good lower bound that involves the signs of $\Sigma$, seems difficult.

\section{Total Graph(s)}\label{sec:Gen}

Recall (say, from \cite[p.~64]{CvDS}) that, for a given graph $G$, the {\em total graph} $\mathcal{T}(G)$  is the graph obtained by combining the adjacency matrix of a graph with the adjacency matrix of its line graph and its vertex-edge incidence matrix. Precisely, the adjacency matrix of $\mathcal{T}(G)$ is given by
 \begin{equation*}\label{eq:Tu(G)}
A_{\mathcal{T}(G)}=\begin{pmatrix}
A_{G}& B\\ B^\intercal &A_{\mathcal{L}(G)}
\end{pmatrix},\end{equation*}
where $\mathcal{L}(G)$ denotes the line graph of $G$ and $B$ is the (unoriented) incidence matrix.
Is it possible to have an analogous concept for signed graphs?

We will give a positive answer to the latter question. However, since we have multiple possibilities for line graphs of signed graphs, we build multiple total graphs.

\subsection{Definitions of a total graph}

\begin{definition}\label{total}
The \textit{total graph} of $\Sigma=G_\sigma$ is the signed graph determined by
\begin{equation}\label{eq:T(G)}
A_{\mathcal{T}_{*}(\Sigma)}=\begin{pmatrix}
A_{\Sigma}& B_\eta\\ B_\eta^\intercal &A_{\mathcal{L}_{*}(\Sigma_\eta)}
\end{pmatrix},\end{equation}
where $*\in\{C,S\}$.
\end{definition}

In other words, in the combinatorial view, $\mathcal{T}_{*}(\Sigma)$ consists of two induced subgraphs, $\Sigma$ and $\mathcal{L}_{*}(\Sigma_\eta)$, along with edges joining a vertex $v$ of $\Sigma$ to all vertices of $\mathcal{L}_{*}(\Sigma_\eta)$ that arise from the edges  which are incident with $v$ in $\Sigma$. The signature on these connecting edges is given by~$\eta$, i.e., for a root-graph vertex $v$ and an incident line-graph vertex $e$, $\sigma_\mathcal{T}(ve) = \eta(v,e)$.  Note that this does not specify an orientation of the edge $ve$.  One may adopt the convention that $\eta_\mathcal{T}(v,ve) = +1$, or any other convention, according to convenience.  We do not need to do that because we have not defined an incidence matrix for $\mathcal{T}_{*}(\Sigma)$.

To fix the notation, $\mathcal{T}_{C}(\Sigma)$ and $\mathcal{T}_{S}(\Sigma)$ denote the total graphs defined by the combinatorial line graph \eqref{eq:BTBZ} and the spectral line graph \eqref{eq:BTBS}. Accordingly, we shall call them respectively the \textit{combinatorial total graph} and the \textit{spectral total graph}. If we want to consider both variants, we shall write $\mathcal{T}_{*}(\Sigma)$ instead.
We need to show that our definition, regardless of $\Sigma\in[\Sigma]$ and of its chosen orientation $\eta$, gives rise to the same signed graph up to switching equivalence.

In Fig.~\ref{fig:T} we illustrate an example of a total graph of a signed graph. The root graph and the orientation are taken from Fig.~\ref{fig:G}.

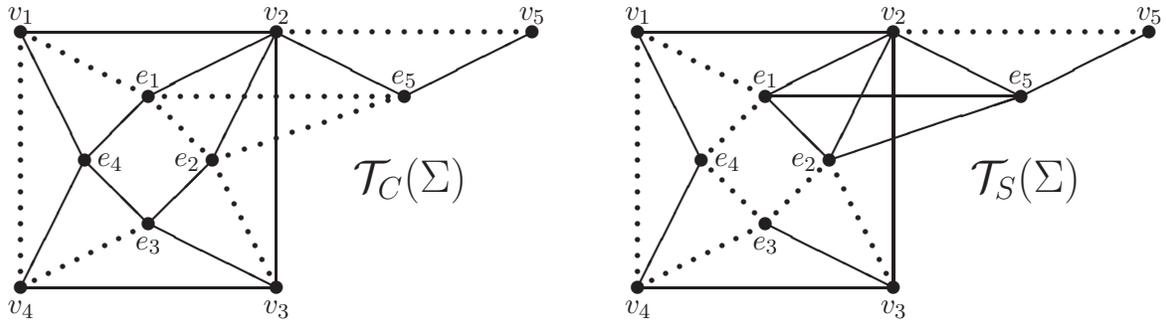
\begin{figure}[H]	\centering
\unitlength=0.85mm
\begin{picture}(80,50)(10,-5)
 \thicklines

\put(0,0){\circle*{2.0}} \put(0,40){\circle*{2.0}} \put(40,40){\circle*{2.0}} \put(40,0){\circle*{2.0}} \put(80,40){\circle*{2.0}}

\put(10,20){\circle*{2.0}} \put(20,10){\circle*{2.0}} \put(20,30){\circle*{2.0}} \put(30,20){\circle*{2.0}} \put(60,30){\circle*{2.0}}

\put(0,0){\line(1,0){40}} \put(0,40){\line(1,0){40}} \put(40,0){\line(0,1){40}}

\multiput(0,0)(0,2.5){16}{\circle*{0.7}} \multiput(40,40)(2.5,0){16}{\circle*{0.7}}

\multiput(0,40)(2.2,-1.1){10}{\circle*{0.7}} \multiput(0,0)(2.2,1.1){10}{\circle*{0.7}}

\multiput(40,0)(-1.1,2.2){10}{\circle*{0.7}} \multiput(20,30)(1.7,-1.7){6}{\circle*{0.7}}

\multiput(20,30)(2.5,0){16}{\circle*{0.7}} \multiput(30,20)(2.5,0.85){13}{\circle*{0.7}}

\put(10,20){\line(-1,2){10}} \put(10,20){\line(-1,-2){10}} \put(10,20){\line(1,1){10}} \put(10,20){\line(1,-1){10}}
\put(60,30){\line(-2,1){20}} \put(60,30){\line(2,1){20}} \put(20,10){\line(1,1){10}} \put(20,10){\line(2,-1){20}}
\put(40,40){\line(-1,-2){10}} \put(40,40){\line(-2,-1){20}}

\put(-2,42){$v_1$} \put(38,42){$v_2$} \put(38,-4){$v_3$} \put(-2,-4){$v_4$} \put(78,42){$v_5$}

\put(18,32){$e_1$} \put(24,19){$e_2$} \put(12,19){$e_4$} \put(18,6){$e_3$} \put(58,32){$e_5$}

\put(52,15){\LARGE{$\mathcal{T}_C(\Sigma)$}}

\end{picture}
\begin{picture}(70,50)(-5,-5)
 \thicklines

\put(0,0){\circle*{2.0}} \put(0,40){\circle*{2.0}} \put(40,40){\circle*{2.0}} \put(40,0){\circle*{2.0}} \put(80,40){\circle*{2.0}}

\put(10,20){\circle*{2.0}} \put(20,10){\circle*{2.0}} \put(20,30){\circle*{2.0}} \put(30,20){\circle*{2.0}} \put(60,30){\circle*{2.0}}

\put(0,0){\line(1,0){40}} \put(0,40){\line(1,0){40}} \put(40,0){\line(0,1){40}}

\multiput(0,0)(0,2.5){16}{\circle*{0.7}} \multiput(40,40)(2.5,0){16}{\circle*{0.7}}

\multiput(0,40)(2.2,-1.1){10}{\circle*{0.7}} \multiput(0,0)(2.2,1.1){10}{\circle*{0.7}}

\multiput(40,0)(-1.1,2.2){10}{\circle*{0.7}}

\multiput(20,30)(-1.7,-1.7){6}{\circle*{0.7}} \multiput(20,10)(-1.7,1.7){6}{\circle*{0.7}} \multiput(20,10)(1.7,1.7){6}{\circle*{0.7}} \put(20,30){\line(1,-1){10}}
\put(20,30){\line(1,0){40}} \put(30,20){\line(3,1){30}}
%\multiput(20,30)(2.5,0){16}{\circle*{0.7}} \multiput(30,20)(2.5,0.85){13}{\circle*{0.7}}

\put(10,20){\line(-1,2){10}} \put(10,20){\line(-1,-2){10}}
%\put(10,20){\line(1,1){10}} \put(10,20){\line(1,-1){10}}
\put(60,30){\line(-2,1){20}} \put(60,30){\line(2,1){20}}
%\put(20,10){\line(1,1){10}}
\put(20,10){\line(2,-1){20}} \put(40,40){\line(-1,-2){10}} \put(40,40){\line(-2,-1){20}}

\put(-2,42){$v_1$} \put(38,42){$v_2$} \put(38,-4){$v_3$} \put(-2,-4){$v_4$} \put(78,42){$v_5$}

\put(18,32){$e_1$} \put(24,19){$e_2$} \put(12,19){$e_4$} \put(18,6){$e_3$} \put(58,32){$e_5$}

\put(52,15){\LARGE{$\mathcal{T}_S(\Sigma)$}}

\end{picture}
	\caption{The combinatorial and the spectral total graphs resulting from $\Sigma_\eta$ depicted in Fig.~\ref{fig:G}.  }\label{fig:T}
\end{figure}

We show that our definitions of a total graph are stable under reorientation and switching.  For reorientation we have the first lemma.

\begin{lemma}\label{le1} Let $\Sigma=G_\sigma$ be a signed graph, and $\Sigma_\eta$ and $\Sigma_{\eta'}$ two orientations of $\Sigma$.  Then $\mathcal{T}_{*}(\Sigma_{\eta})$ and $\mathcal{T}_{*}(\Sigma_{\eta'})$ are switching equivalent, for each $*\in\{C,S\}$.
\end{lemma}

\begin{proof}
For the sake of readability, we will restrict the discussion to the combinatorial line graph defined by Zaslavsky. Hence, hereafter $\mathcal{L}(\Sigma):=\mathcal{L}_C(\Sigma)$ and $\mathcal{T}(\Sigma_{\eta}):=\mathcal{T}_{C}(\Sigma_{\eta})$.

Let $G=(V,E)$, where $|V|=n$ and $|E|=m$. Suppose that $\eta$ and $\eta'$ differ on some set $F\subseteq E$, and let $B_\eta$ and $B_{\eta'}$ be the corresponding vertex-edge incidence matrices, respectively. Let $S=(s_{ij})$ be the $m\times m$ diagonal matrix such that $s_{ii}=-1$ if $e_i\in F$ and $s_{ii}=1$, otherwise. Then $B_{\eta'}=B_\eta S$. Since $S=S^\intercal=S^{-1}$, in view of \eqref{eq:BTBZ}, we have
	$$A_{\mathcal{L}(\Sigma_{\eta'})}=2I-B_{\eta'}^\intercal B_{\eta'} =S^{\intercal}(2I-B_{\eta}^\intercal B_{\eta})S=S^{\intercal}A_{\mathcal{L}(\Sigma_{\eta})}S.$$

Therefore,
	\begin{align*}
	A_{\mathcal{T}(\Sigma_{\eta'})}=&\begin{pmatrix}
	A_{\Sigma} & B_{\eta'}\\ B_{\eta'}^\intercal &  A_{\mathcal{L}(\Sigma_{\eta'}})
	\end{pmatrix}\\=&\begin{pmatrix}
	A_{\Sigma} & B_{\eta}S\\ S^\intercal B_{\eta}^\intercal &  S^\intercal A_{\mathcal{L}(\Sigma_{\eta})}S
	\end{pmatrix}\\=&\begin{pmatrix}
	I & O\\ O &  S^\intercal
	\end{pmatrix}\begin{pmatrix}
	A_{\Sigma} & B_{\eta}\\ B_{\eta}^\intercal &  A_{\mathcal{L}(\Sigma_{\eta})}
	\end{pmatrix}\begin{pmatrix}
	I & O\\ O &  S
	\end{pmatrix}\\=&\begin{pmatrix}
	I & O\\ O &  S
	\end{pmatrix}^{-1}A_{\mathcal{T}(\Sigma_{\eta})}\begin{pmatrix}
	I & O\\ O &  S
	\end{pmatrix}
    .\end{align*}
 This completes the proof.
\end{proof}

Next, we prove that switching equivalent signed graphs produce switching equivalent total graphs.

\begin{lemma}\label{le2}
If $\Sigma$ and $\Sigma'$ are switching equivalent, then $\mathcal{T}_{*}(\Sigma)$ and $\mathcal{T}_{*}(\Sigma')$ are switching equivalent as well, for each $*\in\{C,S\}$.
\end{lemma}

\begin{proof} The notation is the same as in Lemma \ref{le1}. Since $\Sigma$ and $\Sigma'$ are switching equivalent, their adjacency matrices are switching similar. Hence, $A_{\Sigma}=S^{-1}A_{\Sigma'}S$ for some switching matrix~$S$. Observe that if $B=B_{\eta}$ is a vertex-edge incidence matrix of~$\Sigma$, then $B'=SB_{\eta}$ is a vertex-edge incidence matrix of $\Sigma'$. Additionally, in view of \eqref{eq:BTBZ}, we have $A_{\mathcal{L}(\Sigma)}=2I-B^\intercal B$.

Therefore, (for $\Sigma'$) we have
	\begin{align*}
	A_{\mathcal{T}(\Sigma')}=&\begin{pmatrix}
	A_{\Sigma'} & B'\\ B'^\intercal &  A_{\mathcal{L}(\Sigma')}
	\end{pmatrix}=\begin{pmatrix}
	A_{\Sigma'} & B'\\ B'^\intercal & 2I- B'^\intercal B'
	\end{pmatrix}\\=&\begin{pmatrix}
	S^{-1}A_{\Sigma}S & SB\\ (SB)^\intercal & 2I- B^\intercal (S^\intercal S) B
	\end{pmatrix}\\= & \begin{pmatrix}
	S & O\\ O &  I
	\end{pmatrix}^{-1}\begin{pmatrix}
	A_{\Sigma} & B\\ B^\intercal &  A_{\mathcal{L}(\Sigma)}
	\end{pmatrix}\begin{pmatrix}
	S & O\\ O &  I
	\end{pmatrix}\\=&\begin{pmatrix}
	S & O\\ O &  I
	\end{pmatrix}^{-1} A_{\mathcal{T}(\Sigma)}\begin{pmatrix}
	S & O\\ O &  I
	\end{pmatrix} .
	\end{align*}
Hence, $\mathcal{T}(\Sigma)$ is switching equivalent to $\mathcal{T}(\Sigma')$, and we are done.
\end{proof}

In view of Lemmas \ref{le1} and \ref{le2} the definition given by \eqref{eq:T(G)} can be used for spectral investigations.

\begin{remark}
A careful reader has probably noticed that the switching matrix in Lemma~$\ref{le1}$ is the one realizing switching equivalence between the line graphs, while the switching matrix in Lemma~$\ref{le2}$ is the one realizing switching equivalence between the root signed  graphs. In general, if we have two switching equivalent total graphs, then the switching matrix will be obtained by combining the switching matrices of the corresponding root and line graphs.
\end{remark}

\begin{remark}\label{rem:stotal}
The spectral total graph $\mathcal{T}_{S}(\cdot)$ does not generalize the total graph of an unsigned graph, because with the $\mathcal{T}_{S}$ operator the total graph of an all-positive signed graph does not have an all-positive signature, as the convention of treating unsigned graphs as all positive and the definition of the unsigned total graph would imply. On the other hand, if we consider unsigned graphs as signed graphs with the all-negative signature, then $\mathcal{T}_C(-G)=-\mathcal{T}(G)$, so that $\mathcal{T}_C$ can be considered as the generalization to signed graphs of the unsigned total graph operator.
This observation lends some support to using the line graph operator $\mathcal{L}_C$ in spectral graph theory and treating unsigned graphs as all negative, though contrary to existing custom.
\end{remark}

\subsection{Properties of total graphs}

Now we study some structural and spectral properties of $\mathcal{T}_C(\Sigma)$ and $\mathcal{T}_S(\Sigma)$. We begin by computing the number of triangles of $\mathcal{T}_{*}(\Sigma)$.

\begin{theorem}\label{the:tri} Let an unsigned graph $G$ have order $n$, size $m$, degree sequence $(d_1, d_2, \ldots, d_n)$, and $t$ triangles. Then the number of triangles of $\mathcal{T}_{*}(G)$ is $2t+m+\sum_{i=1}^{n}{d_i+1\choose 3}$.
\end{theorem}

\begin{proof}
Every triangle of $\mathcal{T}_{*}(G)$ is one of the following four types:
\begin{enumerate}[\qquad(a)]
		\item belongs to $G$,
		\item belongs to $\mathcal{L}_{*}(G)$,
		\item has 1 vertex in $G$ and 2 vertices in $\mathcal{L}_{*}(G)$,
		\item has 2 vertices in $G$ and 1 vertex in $\mathcal{L}_{*}(G)$.
\end{enumerate}
	Every triangle of $\mathcal{L}_{*}(G)$ arises from a triplet of adjacent edges of $G$. Such a triplet either forms a triangle or has a common vertex. Therefore, $\mathcal{L}_{*}(G)$ contains $t+\sum_{i=1}^{n}{d_i\choose 3}$ triangles.
	
Every triangle of type (c) arises from a pair of adjacent edges of $G$, so their number is $\sum_{i=1}^{n}{d_i\choose 2}$.
	
Every triangle of type (d) arises from an edge of $G$, so their number is $m$.
	
Altogether, the number of triangles is
	$$t+t+\sum_{i=1}^{n}{d_i\choose 3}+\sum_{i=1}^{n}{d_i\choose 2}+m=2t+m+\sum_{i=1}^{n}{d_i+1\choose 3},
	$$
as claimed.
\end{proof}

Moreover, we can establish which triangles are either positive or negative.

\begin{theorem}\label{excor:tri}
Let $\Sigma=G_\sigma$ be a signed graph of order $n$, size $m$, degree sequence $(d_1,d_2,\ldots,d_n)$, and $t=t^{+}+t^-$ triangles, where $t^+$ (resp., $t^-$) denotes the number of positive (resp., negative) triangles. Then $\mathcal{T}_C(\Sigma)$ has exactly $2t^+$ positive triangles, while $\mathcal{T}_S(\Sigma)$ has exactly $t+m$ negative triangles.
\end{theorem}

\begin{proof} The total number of triangles is computed in Theorem \ref{the:tri}.

We have the following facts that the reader can easily check.
The triangles of type (a) will keep their sign in the total signed graph.  Hence, we have $t^+$ positive triangles for $\mathcal{T}_C(\Sigma)$ and $t^-$ negative triangles for $\mathcal{T}_S(\Sigma)$.

A positive (negative) triangle in $\Sigma$ becomes a positive (negative) triangle in  $\mathcal{L}_C(G)$, and a negative (positive) triangle in $\mathcal{L}_S(G)$. A set of mutually adjacent edges in $G$ will give rise to a complete graph in $\mathcal{L}_{*}(G)$ whose signature is equivalent to the all-negative (resp., all-positive) one for $\mathcal{L}_C(\Sigma)$ (resp., $\mathcal{L}_S(\Sigma)$). Summing up, the triangles of type (b) will be $t^+$ positive for $\mathcal{T}_C(\Sigma)$ and $t^+$ negative for $\mathcal{T}_S(\Sigma)$.

Next, let us consider a triangle of type (c). Such a triangle of $\mathcal{T}_{*}(\Sigma)$ is obtained from two edges, say $vu$ and $vw$, of $\Sigma$ that are incident to the same vertex $v$. Regardless of $\sigma(vu)$ and $\sigma(vw)$, we can assign an orientation such that the arrows from the side of $v$ are both inward (directed towards $v$). Hence, the edges $\{v,vw\}$ and $\{v,uv\}$ of $\mathcal{T}_{*}(\Sigma)$ will be positive, while the edge $\{vu,vw\}$ will be negative (resp., positive) in $\mathcal{T}_{C}(\Sigma)$ (resp., $\mathcal{T}_{S}(\Sigma)$).

Finally, a triangle of type (d) comes from a pair of adjacent vertices $u$ and $v$ and the joining edge $uv$. Again, regardless of $\sigma(uv)$ and with a similar reasoning as above, the resulting triangle will always be negative in $\mathcal{T}_{*}(\Sigma)$.

Now, the statement easily follows by counting the positive (negative) triangles of $\mathcal{T}_{C}(\Sigma)$ (resp., $\mathcal{T}_{S}(\Sigma)$).	\end{proof}

\begin{remark}
From Theorem \ref{excor:tri} we easily deduce that $\mathcal{T}_{C}(\Sigma)$ and $\mathcal{T}_{S}(\Sigma)$ have in general switching inequivalent signatures which are not the opposite of each other. Hence, in contrast to the line graphs defined by \eqref{eq:BTBZ} and \eqref{eq:BTBS}, the total graphs derived from them have unrelated signatures.
\end{remark}

We conclude this section by analysing the degree of imbalance of these compound graphs.
A \textit{vertex cover} of a graph is a set of vertices such that every edge has at least one end in the cover.  The smallest size of a vertex cover is the \emph{vertex cover number}, $\tau$.

\begin{theorem}\label{the:1-3} Let $\Sigma=G_{\sigma}$ be a signed graph of order $n$, size $m$, and vertex cover number $\tau$. The following hold true:
	\begin{enumerate}[\rm(i)]

		\item\label{the:1-3bal} $\mathcal{T}_{*}(\Sigma)$ is balanced if and only if $G$ is totally disconnected.
		
		\item\label{the:1-3abal} $\mathcal{T}_{*}(\Sigma)$ is antibalanced if and only if either $* = S$ and $\Sigma$ has no adjacent edges, or $* = C$ and $\Sigma$ is antibalanced.
		
		\item\label{the:1-3fi2} $\ell(\mathcal{T}_{*}(\Sigma)) \geq m + \ell(\mathcal{L}_{*}(\Sigma))$, with equality when $* = S$, and also when $* = C$ and $\Sigma$ is a disjoint union of paths and cycles.
			
		\item\label{the:1-3fi1} $\ell(\mathcal{T}_{*}(\Sigma)) = m$ if and only if either $*=S$ and $\Sigma$ is antibalanced, or $*=C$ and $\Sigma$ is a disjoint union of paths and positive cycles.

		\item\label{the:1-3fn1}  $\nu(\mathcal{T}_{*}(\Sigma))\geq \tau$, with equality if $*=S$ and $\Sigma$ is antibalanced.
		
		\item\label{the:1-3fn2}  $\nu(\mathcal{T}_{S}(\Sigma)) \leq \tau + \nu(\mathcal{L}_{S}(\Sigma))$.
		
		\item\label{the:1-3eig} The largest (adjacency) eigenvalue $\lambda$ of $\mathcal{T}_{*}(\Sigma)$ satisfies
		$$ \lambda\leq\max\Bigg\{ \dfrac{-d_i+\sqrt{5d_i^2+4(d_im_i-4)}}{2}~:~ 1\leq i\leq n+m \Bigg\},	
		$$
		where $d_i$ and $m_i$ denote, respectively, the degree of a vertex $i$ of $\mathcal{T}_{*}(\Sigma)$ and the average degree of its neighbours.
	\end{enumerate}
\end{theorem}

\begin{proof}
\eqref{the:1-3bal}
Each edge of $\Sigma$ leads to a negative triangle of type (d), so $\mathcal{T}_{*}(\Sigma)$ is balanced if and only if $\Sigma$ has no edges.

\eqref{the:1-3abal}
Adjacent edges lead to a positive triangle of type (c) in $\mathcal{T}_{S}(\Sigma)$, hence it cannot be antibalanced.  If there are no adjacent edges, $\mathcal{T}_{S}(\Sigma)$ consists only of negative triangles of type (d) and any isolated vertices of $\Sigma$, which is antibalanced.

There are four kinds of cycle to consider in $\mathcal{T}_{C}(\Sigma)$: triangles of types (c) and (d) and cycles in $\Sigma$ and $\mathcal{L}_{C}(\Sigma)$.  The triangles are negative, hence antibalanced.  If $\Sigma$ is not antibalanced, the total graph cannot be, but if $\Sigma$, thus also $\mathcal{L}_{C}(\Sigma)$ by Theorem \ref{the:lg}\eqref{the:lgSb}, is antibalanced, then it follows -- from the fact that all cycles in the total graph are obtained by combining cycles of those four kinds -- that the total graph is antibalanced.

\eqref{the:1-3fi2}
The $m$ triangles of type (d) of the proof of Theorem~\ref{the:tri} are negative and independent, in the sense that no two of them share the same edge. Therefore, to eliminate each of them  it is necessary to delete $m$ edges (for example, the edges of $\Sigma$ in $\mathcal{T}_{*}(\Sigma)$).
Deleting these edges does not change the frustration index of the subgraph $\mathcal{L}_{*}(\Sigma)$, so at least an additional $\ell(\mathcal{L}_{*}(\Sigma))$ edges must be deleted to attain balance of $\mathcal{T}_{*}(\Sigma)$.
Hence, we have the inequality.

Equality holds for $\mathcal{T}_{S}(\Sigma)$ because the triangles of type (c) are positive.  In $\mathcal{T}_{C}(\Sigma)$ those triangles are negative.  When $\Sigma$ is a disjoint union of paths and positive cycles, then the negative cycles are those of type (c) and (d) which share a common edge. Deleting such independent edges (there are $m$ of them) leads to a balanced signed graph.  When $\Sigma$ has a negative cycle, one edge in those triangles can be replaced by one edge each in the negative cycle in $\Sigma$ and in the corresponding negative cycle in $\mathcal{L}_C(\Sigma)$ for a total of one extra edge for each negative cycle of $\Sigma$.

\eqref{the:1-3fi1}
The equality for $\mathcal{T}_{S}(\Sigma)$ holds under the formulated conditions since there $\mathcal{L}_{S}(\Sigma)$ is balanced and then the entire $\mathcal{T}_{S}(\Sigma)$ becomes balanced after deleting all $m$ edges of $\Sigma$.
For $\mathcal{T}_{C}(\Sigma)$ the result follows from \eqref{the:1-3fi2}.

Conversely, if $\ell(\mathcal{T}_{*}(\Sigma))= m$, then $\mathcal{L}_{*}(\Sigma)$ must be balanced, i.e., it cannot contain a negative cycle. For $*=S$, this means that $\Sigma$ is antibalanced. For $*=C$, this means that $\Sigma$ does not contain a vertex of degree 3 or greater, as the corresponding edges produce negative triangles. Evidently, $\Sigma$ cannot contain negative cycles, because this leads to additional negative cycles in $\mathcal{T}_{C}(\Sigma)$. If $\Sigma$ is a disjoint union of paths and positive cycles, the equality follows from \eqref{the:1-3fi2}.

\eqref{the:1-3fn1}
Consider $\mathcal{T}_{*}(\Sigma)$. For both variants we need to eliminate (at least) the negative triangles of type (d). Instead of deleting the edges of $\Sigma$, we can just delete a minimum vertex cover of $\Sigma$ and obtain the same effect. The equality is obtained, for example, for $\mathcal{T}_{S}(-G)$.

\eqref{the:1-3fn2}
If $B$ is a minimum set of vertices of $\mathcal{L}_{S}(\Sigma)$ such that deleting every vertex in $B$ leaves a balanced line graph $\mathcal{L}_{S}(\Sigma)$, then deleting the same vertices from the line graph in $\mathcal{T}_{S}(\Sigma)$ while also deleting a minimum vertex cover from $\Sigma$ in $\mathcal{L}_{S}(\Sigma)$ as in the proof of \eqref{the:1-3fn1} eliminates all negative cycles in the total graph.

\eqref{the:1-3eig}
Note that, unless $G$ is totally disconnected, every vertex of $\mathcal{T}_{*}(\Sigma)$ belongs to at least one negative triangle -- this triangle is again of type~(d).  Accordingly, the result follows by the inequality of \cite{StaR}:
\begin{eqnarray*} \lambda\leq\max\Bigg\{ \dfrac{-d_i+\sqrt{5d_i^2+4(d_im_i-4t_i^-)}}{2}~:~ 1\leq i\leq n+m\Bigg\},	
\end{eqnarray*}
where $t_i^-$ stands for the number of negative triangles passing through a vertex~$i$.
\end{proof}

\begin{remark}
It is not the case that
$
\nu(\mathcal{T}_{C}(\Sigma)) \leq \tau + \nu(\mathcal{L}_{C}(\Sigma)).
$
A counterexample is a sufficiently long cycle of either sign, for which $\tau \approx \frac12 m$, $\nu(\mathcal{L}_{C}(\Sigma)) \leq 1$, and $\nu(\mathcal{T}_{*}(\Sigma)) \approx \frac23 m \approx \frac43 \tau > \tau + 1$ due to the negative triangles of types (c) and (d).
\end{remark}

\section{Total graphs of regular signed graphs}\label{sec:Reg}

A signed graph $\Sigma={G_\sigma}$ is said to be $r$-regular if its underlying graph $G$ is an $r$-regular graph.

\subsection{Spectra}

From the Perron--Frobenius theorem we infer that the spectrum of $\Sigma$ lies in the real interval~$[-r,\, r]$. We compute the spectrum of $\mathcal{T}_{*}(\Sigma)$ by means of the eigenvalues of the root (signed) graph~$\Sigma$, when it is regular.
%{\color{blue}Perhaps we can keep it between brackets. My feeling is that it is a bit redoundant by there is no possibility of misunderstanding.}

\begin{theorem}\label{the:reg}
	Let $\Sigma$ be an $r$-regular signed graph $(r\geq 2)$ with $n$ vertices and eigenvalues $\lambda_1, \lambda_2, \ldots, \lambda_n$. Then:
\begin{enumerate}[\rm(i)]
\item The eigenvalues of $\mathcal{T}_C(\Sigma)$ are $2$ with multiplicity $(\frac{r}{2}-1)n$ and
	$$\frac{1}{2}\big(2+2\lambda_i-r\pm\sqrt{r^2-4\lambda_i+4}\big),~~\text{for}~~ 1\leq i\leq n.$$
\item The eigenvalues of $\mathcal{T}_S(\Sigma)$ are $-2$ with multiplicity $(\frac{r}{2}-1)n$ and
	$$\frac{1}{2}\big(r-2\pm\sqrt{(r-2\lambda_i)^2+4(\lambda_i+1)}\big),~~\text{for}~~ 1\leq i\leq n.$$
\end{enumerate}
\end{theorem}

\begin{proof}
	The proof is inspired from Cvetkovi\' c's proof of the theorem concerning the total graphs of regular unsigned graphs \cite[Theorem~2.19]{CvDS}. Due to the inconsistency between the concepts of line graphs of unsigned graphs and that of spectral line graphs of signed graphs, our proof differs at some points, as do the final results.
	
	Since $\Sigma$ is $r$-regular, for some incidence matrix $B$, we have $BB^\intercal=L_{\Sigma}=D_G-A_{\Sigma}=rI-A_{\Sigma}$, $A(\mathcal{L}_C(\Sigma))=2I-B^\intercal B$, and $A(\mathcal{L}_S(\Sigma))=B^\intercal B-2I$.

The characteristic polynomial of $\mathcal{T}_C(\Sigma)$ is given by
	\begin{align*}\Phi_{\mathcal{T}_C(\Sigma)}(x)=& \left\lvert\hspace{-1mm}
                                                    \begin{array}{cc} xI-A_\Sigma & -B\\
                                                    -B^\intercal & xI-\mathcal{L}_C(\Sigma)
                                                    \end{array}\hspace{-1mm}\right\rvert\\
                                                 =& \left\lvert\hspace{-1mm}
                                                    \begin{array}{cc} xI-rI+B B^\intercal & -B\\
                                                    -B^\intercal & xI-2I+BB^\intercal
                                                    \end{array}\hspace{-1mm}\right\rvert.
    \end{align*}
	Multiplying the first row of the block determinant by $B^\intercal $ and adding to the second, and then multiplying the second by $\frac{1}{x-2}B$ and adding to the first one, we get
	$$\Phi_{\mathcal{T}_C(\Sigma)}(x)=\left\lvert\hspace{-1mm}\begin{array}{cc}
	(x-r)I+BB^\intercal +\frac{1}{x-2}\big((x-r-1)BB^\intercal +BB^\intercal BB^\intercal \big)& O\\ (x-k-1)B^\intercal +B^\intercal BB^\intercal & (x-2)I\end{array}\hspace{-1mm}\right\rvert.$$
	Further, we compute
	\begin{align*}\Phi_{\mathcal{T}_C(\Sigma)}(x)=&\,(x-2)^{\frac{r}{2}n}\big\lvert
	(x-r)I+BB^\intercal +\frac{1}{x-2}\big((x-r-1)BB^\intercal +BB^\intercal BB^\intercal \big)\big\rvert\\
    =&\,(x-2)^{\frac{r}{2}n} \big\lvert xI-A_{\Sigma}+\frac{1}{x-2}\big((x-r-1)(rI-A_{\Sigma})+(rI-A_{\Sigma})^2\big) \big\rvert\\
    =&\,(x-2)^{(\frac{r}{2}-1)n} \big\lvert A_{\Sigma}^2+(3-2x-r)A_{\Sigma}+(x^2+x(r-2)-r) I \big\rvert\\
    =&\,(x-2)^{(\frac{r}{2}-1)n}\prod_{i=1}^{n}\big(\lambda_i^2+(3-2x-r)\lambda_i+(x^2+x(r-2)-r)\big)\\
    =&\,(x-2)^{(\frac{r}{2}-1)n}\prod_{i=1}^{n}\big(x^2+(r-2-2\lambda_i)x+\lambda_i^2+(3-r)\lambda_i-r\big).\end{align*}
	Since the roots of $x^2+(r-2-2\lambda_i)x+\lambda_i^2+(3-r)\lambda_i-r=0$ are given by $\frac{1}{2}\big(2+2\lambda_i-r\pm\sqrt{r^2-4\lambda_i+4}\big)$, (i) follows.

Item (ii) follows similarly, by taking the spectral variant of the line graph. \end{proof}

From the above theorem we can deduce the real interval containing the eigenvalues of the total graph of a regular signed graph.

\begin{corollary}
Let $\Sigma$ be an $r$-regular signed graph with $n$ vertices and eigenvalues $\lambda_1\geq \lambda_2 \geq \cdots \geq \lambda_n$. Then:
\begin{enumerate}[\rm(i)]
  \item The spectrum of $\mathcal{T}_C(\Sigma)$, if $r\geq4$, lies in the interval
  $$\Big[\frac{1}{2}(2+\lambda_n-r-\sqrt{r^2-4\lambda_n+4}),\ \frac{1}{2}(2+\lambda_1-r+\sqrt{r^2-4\lambda_1+4})\Big].$$
  \item The spectrum of $\mathcal{T}_S(\Sigma)$, if $r\geq2$, lies in the interval
  $$\Big[\frac{1}{2}(r-2-\sqrt{(r-2\lambda_n)^2+4(\lambda_n+1)}),\ \frac{1}{2}(r-2+\sqrt{(r-2\lambda_n)^2+4(\lambda_n+1)})\Big].$$
\end{enumerate}
\end{corollary}

\begin{proof}
Consider first $\mathcal{T}_C(\Sigma)$. It is routine to check that the function $f_1(\lambda)=\frac{1}{2}(2+\lambda-r+\sqrt{r^2-4\lambda+4})$ is increasing for $\lambda\in[-r,\, r]$ when {$r\geq4$}. Hence, the maximum of $f_1$ is attained for $\lambda_1$. The function $f_2(\lambda)=\frac{1}{2}(2+\lambda-r-\sqrt{r^2-4\lambda+4})$ is always increasing, therefore, its minimum is achieved by $\lambda_n$. Hence, the entire spectrum of $\mathcal{T}_C(\Sigma)$ lies in $[\frac{1}{2}(r-2-f_2(\lambda_n)),\ \frac{1}{2}(r-2+f_1(\lambda_n))]$, and we get (i).

Consider next $\mathcal{T}_S(\Sigma)$. Analysing the function
$$f_3(\lambda)=\sqrt{(r-2\lambda)^2+4(\lambda+1)} = \sqrt{(2\lambda-(r-1))^2 + 2r + 3},$$
we find that it is decreasing for $\lambda\leq \frac{r-1}{2}$, increasing for $\lambda\geq \frac{r-1}{2}$, and symmetric around $\frac{r-1}{2}$. Since $\lambda_1\leq r$ and $\lambda_n\leq -1$ (this holds for every signed graph by induction on the number of edges using eigenvalue interlacing), we have $\frac{r-1}{2}-\lambda_n\geq |\frac{r-1}{2}-\lambda_1|$. Since our function is symmetric around $\frac{r-1}{2}$, the last inequality leads to $f_3(\lambda_n)\geq f_3(\lambda_1)$. Hence, $\frac{1}{2}(r-2+f_3(\lambda_n))$ is the largest eigenvalue of $\mathcal{T}_S(\Sigma)$. There are two candidates for the least eigenvalue of the same signed graph: $\frac{1}{2}(r-2-f_3(\lambda_n))$ and  (according to Theorem~\ref{the:reg}(ii)) $-2$. Taking into account that $\lambda_n\leq -1$, we get $\frac{1}{2}(r-2-f_3(\lambda_n))\leq-2$, and thus the least eigenvalue is $\frac{1}{2}(r-2-f_3(\lambda_n))$, which completes~(ii).
 \end{proof}

\subsection{A composition}

We next consider a particular composition of spectral total graphs -- those whose definition is based on the spectral line graph. Of course, similar results can be obtained in case of the combinatorial definition. Some further definitions and notation are needed. The \textit{Cartesian product} (see also \cite{2+zas}) of the signed graphs $\Sigma_1=(G_1, \sigma_1)$ and $\Sigma_2=(G_2, \sigma_2)$ is determined in the following way: (1) Its underlying graph is the Cartesian product $G_1\times G_2$; we state for the sake of completeness that its vertex set is $V(G_1)\times V(G_2)$, and the vertices $(u_1, u_2)$ and $(v_1, v_2)$ are adjacent if and only if either $u_1 = v_1$ and $u_2$ is adjacent to $v_2$ in $G_2$ or $u_2 = v_2$ and $u_1$ is adjacent to $v_1$ in $G_1$. (2) The sign function is defined by
%$$\sigma((u_1, u_2), (v_1, v_2))=\left\{\begin{array}{ll}
%\sigma_1(u_1, v_1)&\text{if}~~u_2=v_2,\\
%\sigma_2(u_2, v_2)&\text{if}~~u_1=v_1.
%\end{array}\right.$$
$$\sigma((u_1, u_2), (v_1, v_2))=\begin{cases}
\sigma_1(u_1, v_1)&\text{if}~~u_2=v_2,\\
\sigma_2(u_2, v_2)&\text{if}~~u_1=v_1.
\end{cases}$$

For the real multisets $\mathcal{S}_1$ and $\mathcal{S}_2$, we denote by $\mathcal{S}_1+\mathcal{S}_2$ the multiset containing all possible sums of elements of $\mathcal{S}_1$ and $\mathcal{S}_2$ (taken with their repetition). Especially, if $\mathcal{S}_2$ consists of the additive identity repeated $i$ times,  then the previous sum  is denoted by $\mathcal{S}_1^i$ and consists of the elements of $\mathcal{S}_1$, each with multiplicity increased by the factor $i$.  Let further $\spec(G_{\sigma})$ denote the spectrum of  $G_{\sigma}$.

Inspired by \cite{Ca5,Ca4}, we consider the polynomial $p(G_{\sigma})=\sum_{i=0}^kc_iG_{\sigma}^i$, where $c_0, c_1, \ldots, c_k$ are non-negative integers ($c_k\neq 0$), $G_{\sigma}$ is a regular signed graph with a fixed orientation, $G_{\sigma}^i$ is the $i$th power of $G_{\sigma}$ with respect to the spectral total graph operation (that is,  $G_{\sigma}^i=\mathcal{T}_S^{i-1}(G_\sigma)$, $i>0$, and $G_{\sigma}^0=K_1$), $c_iG_{\sigma}^i$ denotes the disjoint union of $c_i$ copies of $G_{\sigma}^i$, and the sum of signed graphs is their Cartesian product.

\begin{theorem}For an $r$-regular signed graph $G_{\sigma}$ $(r\geq 2)$ with $n$ vertices,
	\begin{equation}\label{eq:Sp}\spec(p(G_{\sigma}))=\sum_{i=0}^k\spec(G_{\sigma}^i)^{c_i},\end{equation}
	where, for $i\geq 2$,
$\spec(G_{\sigma}^i)$ is comprised of $-2$ with multiplicity $(\frac{r_{i-1}}{2}-1)n_{i-1}$ and
	$$\frac{1}{2}\Big(r_{i-1}-2\pm\sqrt{\big(r_{i-1}-2\lambda_j^{(i-1)}\big)^2+4\big(\lambda_j^{(i-1)}+1\big)}\Big),~~\text{for}~~ 1\leq j \leq n_{i-1},$$
	where
	$r_{i}=2^{i-1}r$, $n_1=n$ , $n_{i}=n\prod_{j=2}^{i}(2^{j-3}r+1),$
and with 		
	$\lambda_1^{(i-1)}, \lambda_2^{(i-1)}, \ldots, \lambda_{n_{i-1}}^{(i-1)}$ being the eigenvalues of $G_{\sigma}^{i-1}$.
	
%{\color{blue} I have deleted an `and'. The list was better to me, but I am fine with this option.}
\end{theorem}

\begin{proof}Since, for a non-negative integer $c$,  $\spec(c\Sigma)=\spec(\Sigma)^c$ and $\spec(\Sigma_1+\Sigma_2)=\spec(\Sigma_1)+\spec(\Sigma_2)$ (for the latter, see \cite{2+zas}), we arrive at \eqref{eq:Sp}, and thus it remains to compute $\spec(\Sigma^i)$. The case $i\leq 1$ is clear and it is also clear that, for $i\geq 2$, $\spec(\Sigma^i)$ is as in the theorem, where $r_{i-1}$ and $n_{i-1}$ are the vertex degree and the number of vertices of $\Sigma^{i-1}$.
	
	By the definition of a  total graph we have $r_{i}=2r_{i-1}$, which along with $r_1=r$ leads to $r_{i}=2^{i-1}r$.
	
	By the same definition, we also have $n_i=n_{i-1}+m_{i-1}$, where $m_{i-1}=\frac{1}{2}n_{i-1}r_{i-1}$ is the number of edges of $\Sigma^{i-1}$. So, we have $n_i=n_{i-1}+\frac{1}{2}n_{i-1}r_{i-1}=n_{i-1}\big(\frac{1}{4}r_{i}+1\big)$. Solving this recurrence, we get
	$$n_i=n\prod_{j=2}^{i}\Big(\frac{1}{4}r_i+1\Big)=n\prod_{j=2}^{i}\big(2^{j-3}r+1\big),$$
	and we are done.
\end{proof}

Regarding the last theorem, for $r=0$ the resulting spectrum is trivial. For $r=1$, $\spec(\Sigma^2)$ is computed directly, not as in the theorem.

\subsection{Eulerian regular digraph}

We conclude the paper by offering a result which applies only to a signed graph that is regular and with all edges positive.  Note that, by the definition, an oriented all-positive signed graph is precisely a directed graph.
 An eigenvalue of a signed graph $G_{\sigma}$ is called \textit{main} if there is an associated eigenvector not orthogonal to the all-1 vector $\mathbf{j}$.  An orientation of a (signed) graph is called \emph{Eulerian} if the in-degree equals the out-degree at every vertex.
\begin{theorem}\label{the:main}
	Let $\Sigma = G_{\sigma}$ be an $r$-regular signed graph with the all-positive signature and an Eulerian orientation $\eta$. Then $\mathcal{T}_S(\Sigma_\eta)$ has exactly two main eigenvalues: $r$ and $-2$.
\end{theorem}

\begin{proof}
The assumptions of positive edges and Eulerian orientation mean that the row sums of $B_\eta$ are zero.
Under our assumptions, every block of \eqref{eq:T(G)} has a constant row sum given in the following (quotient) matrix
	$$Q=\begin{pmatrix}
	r&0\\ 0& -2
	\end{pmatrix}.$$
	The spectrum of $Q$ contains the main part of the spectrum of $\mathcal{T}_S(\Sigma_\eta)$. (The explicit proof can be found in \cite{StaM}, but the reader can also consult \cite[Chapter~4]{CvDS} or \cite{Ati}.) By definition, every signed graph has at least one main eigenvalue. If $\mathcal{T}_S(\Sigma_\eta)$ has exactly one main eigenvalue, then the eigenspace of any other is orthogonal to $\mathbf{j}$, which implies that $\mathbf{j}$ is associated with the unique main eigenvalue, but this is impossible (for $\mathcal{T}_S(\Sigma_\eta)$, observe $Q$). Therefore, both eigenvalues of~$Q$ are the main eigenvalues of $\mathcal{T}_S(\Sigma_\eta)$, and we are done.
\end{proof}

%Eulerian directed graphs can easily be constructed. Here is an example.
%
%\begin{example}\label{exa1} By taking a regular graph whose edge set is decomposable into cycles and choosing a biorientation whose restriction on every cycle of a decomposition satisfies the assumption of Theorem~\ref{the:main}, we arrive at a desired oriented signed graph. An example is illustrated in Fig.~\ref{fig:exa}.
%	
%\end{example}
%
%\begin{figure}[H]
%	\centering
%	\includegraphics[width=60mm,angle=0]{exa.pdf}
%	\caption{The oriented signed graph of Example~\ref{exa1} in which positive vertex-edge orientations are indicated. Its edges decompose into the cycles $12345678$, $1526$ and $3748$.}\label{fig:exa}
%\end{figure}

\section*{Acknowledgements}

The research of the second author is supported by the Serbian Ministry of Education, Science and Technological Development via the University of Belgrade.

\end{document}